\documentclass[letterpaper, 11pt]{article}
\usepackage[numbers]{natbib}
\usepackage{thm-restate}
\usepackage[utf8x]{inputenc}
\usepackage[T1]{fontenc}
\usepackage{lmodern}
\usepackage[hidelinks,hypertexnames=false]{hyperref}
\usepackage[numbers]{natbib}
\usepackage[margin=1in]{geometry}
\usepackage{multicol}
\usepackage{booktabs}
\usepackage{amsfonts} \usepackage{amssymb} \usepackage{amsmath}
\usepackage{graphicx} 
\usepackage{epstopdf} 

\usepackage{amsthm} \usepackage{enumerate}
\usepackage{url} \usepackage{verbatim} \usepackage{float} 

\usepackage{tabularx,ragged2e,booktabs,caption} \usepackage{makecell}
\usepackage[plain]{algorithm2e}
\usepackage{enumitem} 
\usepackage{caption}

\usepackage{float}
\usepackage{subfig}
\usepackage{makecell}
\usepackage{mathrsfs}
\usepackage{tikz}
\usepackage{multirow}
\usepackage{chngcntr}
\usepackage{mathtools}
\usepackage{cleveref}
\counterwithin{table}{subsection}

\DeclareMathOperator*{\bigO}{\ensuremath{\mathcal{O}}}

\DeclareMathOperator*{\rank}{rank}

\DeclareMathOperator*{\minimize}{minimize}
\DeclareMathOperator*{\maximize}{maximize}

\newcommand{\ones}{\mathbf{1}}

\newcommand{\algName}{\mathbb{R}^m}

\newtheorem{thm}{Theorem}
\numberwithin{thm}{section}

\numberwithin{cor}{section}

\newtheorem{lem}{Lemma}
\numberwithin{lem}{section}

\newtheorem{prop}{Proposition}
\numberwithin{prop}{section}

\newtheorem{defn}{Definition}
\numberwithin{defn}{section}

\numberwithin{ex}{section}

\newtheorem{ass}{Assumption}

\newcommand*{\newton}{d}

\newcommand{\qup}{q}

\title{Log-domain interior-point methods for convex quadratic programming}
\author{Frank Permenter}
\begin{document}
\maketitle
\abstract{
  Applying an interior-point method  to the central-path conditions is a widely
  used approach for solving quadratic programs.  Reformulating these conditions
  in the log-domain is a natural variation on this approach that to our
  knowledge is previously unstudied.  In this paper, we analyze log-domain
  interior-point methods and prove their polynomial-time convergence.
  We also prove that they are approximated by classical barrier methods in a
  precise sense and provide simple computational experiments illustrating their
  superior performance.
}

\section{Introduction}
Interior-point methods (IPMs) are widely used numerical
algorithms for solving convex quadratic programs (QPs)
of the form
\begin{align}~\label{qp:main}
  \begin{aligned}
    \mbox{minimize } &\quad \frac{1}{2}x^T W x + c^{T}x\\
    \mbox{subject to} &\quad Ax + b \ge 0,
  \end{aligned}
\end{align}
where $x \in \mathbb{R}^n$ is the decision variable,
$A \in \mathbb{R}^{m \times n}$ and $b \in \mathbb{R}^m$ define linear inequality constraints, 
and $W \in \mathbb{R}^{n \times n}$ is a symmetric, positive semidefinite matrix 
that, together with $c \in \mathbb{R}^n$, defines a convex, quadratic objective.
IPMs solve~\eqref{qp:main} by tracking 
the solution $(x, s, \lambda) \in \mathbb{R}^n \times \mathbb{R}^{m} \times \mathbb{R}^m$
to the \emph{central-path} conditions
\begin{align}\label{eq:cp}
  A^{T} \lambda = Wx + c, \;\; s = Ax + b, \qquad 
  \lambda \ge 0, s \ge 0,  \;\;s_i \lambda_i = \mu \;\; \forall i \in \{1, 2,\ldots, m\}
\end{align}
for a decreasing sequence of $\mu > 0$.  When $\mu = 0$, these  are precisely the Karush-Kuhn-Tucker
(KKT) optimality conditions for~\eqref{qp:main}. Hence, 
by gradually reducing $\mu$ to zero, IPMs produce an optimal solution $x$ to~\eqref{qp:main}
along with an optimal constraint slack $s$ and corresponding vector $\lambda$
of Lagrange multipliers.  IPMs are efficient in practice
and have several high quality implementations~\cite{optimization2012gurobi,
mosek2010mosek}. They are also efficient in theory,
 requiring just $\bigO(\sqrt m)$ iterations to solve the QP~\eqref{qp:main} to fixed accuracy,
where the per-iteration cost  is the solution of an $n \times n$ linear
system~\cite{wright1997primal, terlaky2013interior, achache2006new}.  

Success of IPMs requires 
existence and uniqueness of the central path, i.e., of solutions $(x, s, \lambda)$
to~\eqref{eq:cp} for all $\mu > 0$. Using standard arguments (e.g.,~\cite[Theorem~1]{grana2000central}),
this holds by further assuming the QP~\eqref{qp:main} satisfies
the following conditions.
\begin{ass}\label{ass:main}
  The following conditions hold:
  \begin{itemize}
    \item There exist $x\in\mathbb{R}^n$ and $s\in\mathbb{R}^m$ with $s> 0$ satisfying $s = Ax + b$.
    \item For all $\beta \in \mathbb{R}$, the sublevel set $\{ x \in \mathbb{R}^n : \frac{1}{2}x^{T}Wx + c^{T}x \le \beta, Ax + b \ge 0 \}$ is bounded.
    \item $A^T A + W \succ 0$, i.e.,  $A^T A + W$ is positive definite. 
  \end{itemize}
\end{ass}
\noindent We assume that these conditions hold throughout.
Note that these conditions impose no direct constraint  on the shape of $A \in
\mathbb{R}^{m \times n}$, i.e., we may have that $m = n$, $m < n$, or $m>n$.
Further, if $W = 0$, then the condition $A^T A \succ 0$
is the usual assumption for linear programming that the constraint matrix of $Ax + b \ge 0$
is full column rank; see, e.g.,~\cite{andersen2003implementing}.


\subsection{Log-domain interior-point methods}~\label{sec:ipm_template}
The set of nonnegative $(s, \lambda)$ satisfying $s_i \lambda_i = \mu$ for $i \in \{1, 2, \ldots, m\}$
is easily parameterized in the log-domain:
letting $e^v \in \mathbb{R}^m$ denote elementwise exponentiation,
this condition holds if and only if $\lambda = \sqrt \mu e^{v}$ and $s = \sqrt \mu e^{-v}$
  for some $v\in \mathbb{R}^m$.
  This $v$-parametrization of $s$ and $\lambda$ yields the following
   log-domain reformulation of the central-path conditions~\eqref{eq:cp} 
\begin{align}~\label{eq:logcentral}
\sqrt{\mu} A^T   e^v = Wx + c, \;\; \sqrt{\mu} e^{-v}  =  Ax + b
\end{align}
and a template \emph{log-domain interior-point method} for solving the QP~\eqref{qp:main}:
  \begin{itemize}
    \item Update $(v, x)$ by applying Newton's method to~\eqref{eq:logcentral} for fixed $\mu$.
    \item Reduce $\mu$ and repeat.
  \end{itemize}
This paper studies this template algorithm, which, to our knowledge,
has not previously appeared in the QP literature. As we show, there exist 
concrete instantiations that are both practically and theoretically efficient.
In particular, we provide a \emph{short-step} algorithm and prove the typical
$\bigO(\sqrt m)$ iteration bound.  We also provide a \emph{long-step} version
and illustrate its practical performance.  

\subsection{Prior work}

The literature on quadratic programming is vast and we will not attempt to cite
it completely. We do note that IPMs with polynomial-time complexity trace to
\cite{karmarkar1984new, kapoor1986fast} and IPMs with $\bigO(\sqrt
m)$ iteration bounds include~\cite{monteiro1989interiorQuad, goldfarb1990n, goldfarb1993n, kojima1991sqrt}.  
One can also obtain a $\bigO(\sqrt m)$ bound by invoking the
\emph{self-concordance} of a suitable \emph{barrier function};
see~\cite{nesterov1994interior}.
For linear objectives ($W=0$), our algorithms are special cases of 
\emph{geodesic interior-point methods}~\cite{permenter2020geodesic},  recent
techniques for minimizing a \emph{linear} function  subject to \emph{symmetric
cone} inequalities. Indeed, our main analysis task is showing that key
convergence results of~\cite{permenter2020geodesic} still hold 
when a quadratic objective term $x^{T}Wx$ is included.

Linear updates of the log parameter $v$ 
are of course multiplicative updates
of $s = \sqrt\mu e^{-v}$  and $\lambda = \sqrt\mu e^{v}$.
Algorithms based on multiplicative updates
have been developed for restricted families of QPs, e.g.,
linear programs $(W = 0)$, nonnegative least-squares $(A = I, b = 0)$,
and model-predictive control; see, e.g.,~\cite{arora2012multiplicative, sha2007multiplicative, di2013multiplicative}. 
We emphasize that in each of these algorithms, 
the updates are distinct from ours and are designed in different ways. In
particular, they are only applied to one variable, $\lambda$ or $s$, and are
not based on the log-transformation~\eqref{eq:logcentral} of the central-path
conditions.





\subsection{Outline of contributions}
The contributions of this paper are organized as follows. \Cref{sec:Newton} analyzes the application
of Newton's method to the log-domain central-path equations~\eqref{eq:logcentral},
establishing a globally-convergent step-size rule and a local region of
quadratic convergence.  Building on this analysis, Section~\ref{sec:alg}
provides two algorithms  for solving the QP~\eqref{qp:main}
based on two different $\mu$-update rules. The first is a short-step
algorithm that reduces $\mu$ at a fixed rate and  terminates after at most
$\bigO(\sqrt m)$ Newton iterations.  The second is a long-step algorithm that
employs more aggressive $\mu$-updates via line-search.
Section~\ref{sec:barrier} provides theoretical and computational
comparisons with barrier methods. In particular, we show that these methods
are, in a precise sense, approximations of the presented log-domain IPMs. 

%
%

\section{Newton's method}~\label{sec:Newton}
Applying Newton's method to the log-domain central-path
equations~\eqref{eq:logcentral} proceeds by Taylor-approximating the
exponential functions $e^v$ and $e^{-v}$.
Letting $x \circ y$ denote elementwise multiplication of $x, y \in \mathbb{R}^m$,
these approximations take the form
\begin{align*}
  \begin{aligned}
    e^{v + d} &\approx  e^v +  e^v \circ d, \\
  \end{aligned}
  \qquad
  \begin{aligned}
    e^{-(v + d)} &\approx  e^{-v} - e^{-v}\circ  d.
\end{aligned}
\end{align*}
The Newton direction  $\newton(v, \mu)\in\mathbb{R}^m$ and an associated~$x(v,
\mu) \in \mathbb{R}^n$ are then defined by substituting
these approximations into~\eqref{eq:logcentral}.
\begin{defn}\label{defn:NewtonDirection}
  For fixed $\mu > 0$ and $v \in \mathbb{R}^m$,
  the \emph{Newton direction} $\newton(v, \mu)$ 
  and associated  $x(v, \mu)$ are the $d \in \mathbb{R}^m$ and $x \in \mathbb{R}^n$
  satisfying
  \begin{align*}
    \sqrt \mu   A^T(e^v +  e^v \circ d) =  Wx + c,   \;\;      \sqrt \mu   (e^{-v} -  e^{-v} \circ d)  = Ax + b.
  \end{align*}
\end{defn}
Note that when $d(v, \mu) = 0$,
we obtain a point  $(x, s, \lambda)$ on the central-path
by taking $x = x(v, \mu)$,  $s =\sqrt\mu e^{-v}$, and $\lambda= \sqrt\mu e^v$.
Further, $x(v, \mu)$ is a ``good'' approximate
solution of the QP~\eqref{qp:main} when $\mu$ and $\|d(v, \mu)\|$ are sufficiently
small. (Exact error bounds will be given in Section~\ref{sec:longstep}.)
It remains to prove that $\newton(v, \mu)$  and $x(v, \mu)$ are well-defined
for all $\mu >0$ and $v \in \mathbb{R}^m$.
For this, we next show that $\newton(v, \mu)$ is a function of $x(v, \mu)$
and that $x(v, \mu)$ is the unique solution of a consistent linear system.
In particular, we show that this linear system is of the form $Sx = f$ for $S \succ 0$, i.e.,
for $S$ symmetric and positive definite. 
\begin{thm}\label{thm:NewtonChar}
  For all $v \in \mathbb{R}^m$ and $\mu > 0$, the Newton direction $\newton(v, \mu)$ and point $x(v, \mu)$ satisfy
  \[
  \newton = \ones -  \frac{1}{\sqrt\mu} e^v \circ (Ax + b),
  \]
where $\ones \in \mathbb{R}^m$ denotes the vector of all ones. 
  Moreover, $x(v, \mu)$ is the unique solution of
  \[ 
  (A^T Q(v) A +  W)x = 2\sqrt \mu  A^{T} e^v  -(c+ A^{T}Q(v)b),
  \]
where $Q(v) \in \mathbb{R}^{m \times m}$ is the diagonal matrix with $[Q(v)]_{ii} = e^{2v_i}$. 
Further, $A^T Q(v) A +  W \succ 0$.

  \begin{proof}
  Rearranging $\sqrt \mu   (e^{-v} -  e^{-v} \circ d)  = Ax + b$, we conclude that
\begin{align*}
  d &=  \ones -  \frac{1}{\sqrt\mu} e^v \circ (Ax + b).
\end{align*}
Substituting into $\sqrt \mu   A^T(e^v +  e^v \circ d) =  Wx + c$ yields
\begin{align*}
  Wx + c &= \sqrt \mu A^{T} e^v \circ (\ones+d) \\
   &=\sqrt \mu  A^{T} e^v \circ ( \ones + \ones - \frac{1}{\sqrt\mu} e^v \circ ( Ax + b)).
\end{align*}
Rearranging and using $Q(v)Ax = e^v \circ (e^v \circ Ax)$ shows that
\[
  (A^T Q(v)  A +  W) x  =     2\sqrt \mu  A^{T} e^v  -(c+ A^{T}Q(v)b).
\]
 Uniqueness of $x$ follows because  $A^T Q(v)  A +  W$ is 
    positive definite under our assumption that $A^{T}A + W$ is positive definite (\Cref{ass:main}).
    To see this, suppose that $(A^T Q(v)  A +  W) z = 0$ for nonzero $z$.
    Then, $Wz = 0$ and  $A^T Q(v)  A  z = 0$.  But this implies that $Q(v)^{1/2}  A  z = 0$, which,
    in turn means that $Az = 0$ since $Q(v)^{1/2}$ is invertible. We conclude that $(A^T A + W)z = 0$,
    a contradiction.
  \end{proof}
\end{thm}

The remainder of this section studies convergence of the Newton iterations
\[
  v_{i+1} = v_i + \frac{1}{\alpha_i} d(v_i, \mu)
\]
under a simple step-size rule  for choosing  $\alpha_i \in \mathbb{R}$.
We will show global convergence to \emph{centered points}.
\begin{defn}[Centered points]
  For $\mu > 0$, the \emph{centered point} $\hat v(\mu)$ is the $v\in\mathbb{R}^m$ that, for some $x\in\mathbb{R}^n$, solves the log-domain
central-path equations~\eqref{eq:logcentral}.
\end{defn}
\noindent Following~\cite{permenter2020geodesic}, we will measure the distance
of an iterate $v_i$ to $\hat v(\mu)$ using \emph{divergence}.
\begin{defn}[Divergence~\cite{permenter2020geodesic}]
  The \emph{divergence} $h(u, v)$ of $(u, v) \in \mathbb{R}^m \times \mathbb{R}^m$
  is
\[
    h(u, v) := \langle e^{u}, e^{-v}\rangle  + \langle e^{-u}, e^{v}\rangle - 2m.
\]
  For fixed $\mu > 0$, the function $h_{\mu} : \mathbb{R}^m \rightarrow \mathbb{R}$ 
  denotes the map $v \mapsto h(\hat v(\mu), v)$. 
\end{defn}
\noindent While divergence is \emph{not} a metric,
it does have a set of properties useful for convergence analysis.
\begin{lem}\label{lem:divprop}
  The following properties hold for all $u, v \in \mathbb{R}^m$ and $\mu > 0$.
  \begin{enumerate}[label= (\alph*)]
    \item  $h(u, v) = h(v, u)$ and $h(u, v) \ge 0$.
    \item\label{item:divposdef}  $h(u, v) = 0$ if and only if $u = v$. In particular,
      $h(u, v) = -2m + \sum^m_{i=1} 2\cosh(v_i - u_i)$.
    \item\label{item:divstrongconvex}$h_{\mu} : \mathbb{R}^m \rightarrow \mathbb{R}$ is strongly convex. In particular, $\frac{1}{2}\nabla^2 h_{\mu}(v) \succeq I$.
  \end{enumerate}
\end{lem}
\noindent Leveraging these properties, Section~\ref{sec:stepsize} 
provides a step-size rule for which $h_{\mu}(v_{i}) < h_{\mu}(v_{i-1})$ for all iterations $i$.
Building on this, Section~\ref{sec:globalconv}
shows that the sequence $v_0, v_1, v_2,\ldots$ converges to the centered point $\hat v(\mu)$ 
from an arbitrary initial $v_0 \in \mathbb{R}^m$.
Finally, Section~\ref{sec:quadconv} shows
quadratic convergence  when $h_{\mu}(v_0) \le \frac{1}{2}$.
As we will point out, some statements generalize previous
results for linear optimization~\cite{permenter2020geodesic} 
to the quadratic program~\eqref{qp:main}.

\subsection{Step-size rule}~\label{sec:stepsize}
Our step-size rule arises from bounds on the directional
derivatives of divergence $h_{\mu}(v)$.
Towards stating them, fix $v \in \algName$ and  $\mu > 0$
and for brevity let $d\in\mathbb{R}^m$ denote the Newton direction $\newton(v, \mu)$.
Assume that $d \ne 0$ or, equivalently, that $v \ne \hat v(\mu)$.
Finally, let $f : \mathbb{R} \rightarrow \mathbb{R}$ 
denote the restriction of $h_{\mu}(v)$ to the line induced by $v$ and $d$, i.e.,
\[
  f(t) := h_{\mu}(v + t d).
\]
The next lemma provides bounds on $f'(0)$ and $f''(t)$ and
generalizes~\cite[Lemma
3.4]{permenter2020geodesic}~and~\cite[Lemma 3.6]{permenter2020geodesic}.  
\begin{lem}\label{lem:NewtonDirDerivs}
The following statements hold.
  \begin{itemize}
    \item $f{'}(0) \le  -(f(0) + \|d\|^2)$.
    \item $f^{''}(t) \le \|d\|^{2}_{\infty} f(t) + 2 \|d\|^2$.
    \item For all intervals $[a, b]\subset \mathbb{R}$, we have $\sup_{\zeta \in [a,b]} f''(\zeta) \le  \max_{\zeta \in \{a, b\}} (\|d\|^{2}_{\infty} f(\zeta) + 2 \|d\|^2)$.
  \end{itemize}
  \begin{proof}
    For brevity, let $w = e^v$, $\hat w = e^{\hat v(\mu)}$ and $k = \sqrt \mu$.
Letting $z^{-1}$ denote elementwise inversion,
    define $p := \langle w^{-1} \circ (\ones- d) -  \hat w^{-1},   w \circ(\ones+d)- \hat w \rangle$.
Expanding, we conclude that
\begin{align*}
  p &= \langle (\ones -d), (\ones +d) \rangle -\langle w^{-1} \circ \hat w, (\ones -d) \rangle -\langle w \circ\hat w^{-1}, (\ones +d) \rangle
+m \\
  &=m - \|d\|^2  -\langle w^{-1} \circ \hat w, (\ones -d) \rangle -\langle w \circ \hat w^{-1}, (\ones +d) \rangle + m\\
  &= -( \langle  \hat w, w^{-1} \rangle + \langle  \hat w^{-1}, w \rangle - 2m)   - \|d\|^2  -
  \langle w \circ \hat w^{-1} - w^{-1} \circ \hat w, d \rangle
\\
  &= -f(0) - \|d\|^2 - f'(0).
\end{align*}
We now show that $p \ge 0$ for the Newton direction $d$, which will prove the first statement.  Since $\hat w$ solves~\eqref{eq:cp}, we have for some $\hat x$, that
  $b = k \hat w^{-1} - A\hat x$ and   $c = kA^{T}\hat w  - W\hat x$.
Substituting these expressions for $b$ and $c$ into the definition of $d$, we have,
for some $x$, that
    \[
      k w^{-1} \circ (\ones- d) - k \hat w^{-1}   = A(x - \hat x) , \;\;
      k A^T (w \circ (\ones+ d) -  \hat w) = W(x - \hat x).
    \]
Using the first equation, we  conclude that $p = \langle \frac{1}{k}A(x-\hat
    x),  w \circ(\ones+d)- \hat w  \rangle$.  Combining with the second yields
    $p  = \langle \frac{1}{k}(x - \hat x), \frac{1}{k} W(x-\hat x) \rangle \ge
    0$, which proves the first statement.  Proof of the second statement is identical to~\cite[Lemma
    3.4.c]{permenter2020geodesic}.  The third statement follows from the second
    and convexity of $f(t)$.
  \end{proof}
\end{lem}
Combining this lemma with the inequality 
\begin{align}~\label{eq:taylorup}
f(t) \le f(0) + f'(0) t + \frac{1}{2} \sup_{\zeta \in [0, t]} f''(\zeta) t^2
\end{align}
yields a piecewise step-size rule for selecting $t$ such that $f(t) < f(0)$. This rule
is parameterized by $0 < \beta < 1$ which, along with 
$\|d\|_{\infty}^2$, controls the transition from full to damped steps. 
\begin{thm}~\label{thm:newtonstep}
  For $\beta \in (0, 1)$,  let $\alpha  = \max(1, \frac{1}{2\beta} \|d\|_{\infty}^2)$. 
  The following statements hold.
  \begin{enumerate}[label= (\alph*)]
    \item~\label{item:thm:newton:strict}$f(\frac{1}{\alpha}) < f(0)$
    \item~\label{item:thm:newton:full}If $\alpha=1$, then $f(1) \le \frac{1}{2} \|d\|^2_{\infty} f(0) \le \beta f(0)$.
  \end{enumerate} 
  \begin{proof}
    Let $\hat t \ge 0$  denote the smallest $t$ for which $f(t) = f(0)$. By
    strong convexity (\Cref{lem:divprop}), we have that $f(t) < f(0)$ for all $t \in (0, \hat t)$
    since $d$ is a descent direction (\Cref{lem:NewtonDirDerivs}). Further $\hat t > 0$.
    Towards bounding $\hat t$, we first note that the combination of~\eqref{eq:taylorup} with Lemma~\ref{lem:NewtonDirDerivs}
    implies that for all $t$,
    \begin{align}\label{eq:newtonProof}
      f(t) \le f(0) - t(f(0) + \|d\|^2) + \frac{1}{2}(\|d\|^2_{\infty} \max(f(0), f(t)) + 2\|d\|^2) t^2.
    \end{align}
    Substituting $t = \hat t$ and using $f(0) = f(\hat t)$, we conclude that
    \[
      \hat t (\frac{\|d\|^2_{\infty}}{2} f(0) + \|d\|^2 ) \ge f(0) + \|d\|^2.
    \]
    Hence, $t < \hat t$ if $t (\frac{\|d\|^2_{\infty}}{2} f(0) + \|d\|^2 ) < f(0) + \|d\|^2$,
    which holds if $t  = \min(1, \frac{2\beta}{\|d\|_{\infty}^2})$,
    proving~\cref{item:thm:newton:strict}.
    \Cref{item:thm:newton:full} follows by substituting $t = 1$ 
    and $\max(f(0), f(1))  = f(0)$  into~\eqref{eq:newtonProof}.
  \end{proof}
\end{thm}
%

\subsection{Global convergence}~\label{sec:globalconv}
Newton iterations strictly decrease the divergence $h_{\mu}(v)$ under the
step-size rule of~\Cref{thm:newtonstep}.  Combined with  the strong convexity
of $h_{\mu}$, this implies  convergence to the centered point $\hat v(\mu)$.
\begin{thm}\label{thm:globalconv}
  Fix $0 < \beta < 1$ and $\mu > 0$.  For all $v_0 \in \mathbb{R}^m$, 
the Newton iterations $v_{i+1} = v_{i} + \frac{1}{\alpha_i} d(v_i, \mu)$ with step-size rule $\alpha_i =\max\{1, \frac{1}{2\beta}\|d(v_i, \mu)\|^2_{\infty}\}$
  converge to the centered point $\hat v(\mu)$. 
  \begin{proof}
    By choice of $\alpha_i$ and
    \Cref{thm:newtonstep}-\ref{item:thm:newton:strict}, we have that
    $h_{\mu}(v_i)$ strictly decreases.
    In addition, $h_{\mu}(v_i) \ge 0$ for all $v_i$; hence, it converges to some nonnegative number $\delta$.
    We will show that $\delta = 0$, which implies that $v_i$ converges to $\hat v(\mu)$ by \Cref{lem:divprop}-\ref{item:divposdef}.
    To begin, note that all iterations $v_i$ are contained in the set
      $\Omega := \{ v \in \mathbb{R}^m : \delta \le h_{\mu}(v) \le h_{\mu}(v_0) \}$.
    But $\Omega$ is compact since $h_{\mu}$ is strongly convex (\Cref{lem:divprop}).
    Letting $\alpha(v) := \max\{1, \frac{1}{2\beta}\|d(v, \mu)\|^2_{\infty}\}$,
    we conclude that the continuous function
      $D(v) :=  h_\mu(v) - h_\mu(v + \frac{1}{\alpha(v)} d(v, \mu))$
    obtains its infimum $D^*\in\mathbb{R}$ on $\Omega$.  But if $\delta > 0$,
    then $D^* > 0$, which implies that $h_{\mu}(v_m) \le h_{\mu}(v_0) - m D^* < 0$
    for all $m >  h_{\mu}(v_0)/D^{*}$, a contradiction since $h_{\mu}(v_m)  \ge
    0$.  Hence, $\delta = 0$. 
  \end{proof}
\end{thm}

\subsection{Local quadratic convergence}~\label{sec:quadconv}
 \Cref{thm:newtonstep} states that  a full Newton-step $v_{i+1} = v_i + d(v_i, \mu)$
 decreases the divergence $h_{\mu}(v_i)$ by a factor of at least
 $\frac{1}{2}\|d(v_i, \mu)\|^2_{\infty}$.  The next lemma, which generalizes~\cite[Corollary 3.1]{permenter2020geodesic},
 shows that we can also upper-bound $\|d(v_i, \mu)\|^2$, and hence $\|d(v_i, \mu)\|^2_{\infty}$,
 using $h_\mu(v_i)$.  
\begin{lem}~\label{lem:divBound}
  For all $v \in \mathbb{R}^m$ and $\mu > 0$, it holds that $  \|d(v, \mu) \|^2
  \le h_{\mu}(v) (1+\|d( v, \mu) \|)$. 
  \begin{proof}
    For brevity, let $d$ denote $d( v, \mu)$ and let $a = e^{v-\hat v(\mu)}$ and $g = a - a^{-1}$.
    As in Section~\ref{sec:globalconv}, let $f(t) = h_{\mu}(v + t d)$ such that $f(0) = h_{\mu}(v)$.
    Observing that $g$ is the gradient of $h_{\mu}(v)$ with respect to $v$, we have,
    by Lemma~\ref{lem:NewtonDirDerivs} and Cauchy-Schwartz, that
    \[
      f(0) + \|d\|^2  \le |f'(0)| = |\langle g, d \rangle| \le \|g\| \|d\|.
    \]
    We also have that $\|g\|^2 \le f(0)^2 + 4f(0)$ given that
\[
  f(0) = \|a+a^{-1} - 2\ones \|_1  \ge \|a+a^{-1} - 2\ones \|_2 =  \sqrt{\|a - a^{-1} \|^2 - 4 \langle \ones, a + a^{-1} - 2\ones\rangle} =  \sqrt{ \|g\|^2 - 4 f(0)}.
\]
    We conclude that $f(0) + \|d\|^2 \le \|d\| \sqrt{f(0)^2 + 4f(0)}$.
    Squaring each side and rearranging yields
    \begin{align*}
      0 \le \|d\|^2 (f(0)^2 + 4f(0)) - (f(0) + \|d\|^2)^2 = (-\|d\|^2 + f(0)(1+\|d\|)) (\|d\|^2 + f(0)(\|d\|-1)).
    \end{align*}
    This shows that if $f(0)(1+\|d\|) < \|d\|^2$, then $\|d\|^2  \le  f(0)(1-\|d\|)$,
    which in turn implies that
    \[
      f(0)(1+\|d\|)  < f(0)(1-\|d\|),
    \]
    which is impossible. Hence, $f(0)(1+\|d\|) \ge \|d\|^2$, as desired.




%
  \end{proof}
\end{lem}
\noindent Combining this with \Cref{thm:newtonstep}
yields the following quadratic convergence result,
which generalizes~\cite[Theorem 3.4]{permenter2020geodesic}.
\begin{thm}~\label{thm:newtonConv}
  For $\mu >0$ and $v_0 \in \mathbb{R}^m$, 
  let $v_{i+1} =
  v_i + \newton(v_i, \mu)$.  If $h_{\mu}(v_0) \le \theta \le \frac{1}{2}$, then
    $h_{\mu}(v_i) \le   \theta^{2^i}$.
  \begin{proof}
    Let $h_i = h_{\mu}(v_i)$ and $d_i =\newton(v_i, \mu)$.  
    Make the inductive hypothesis that $h_i \le \frac{1}{2}$. Then Lemma~\ref{lem:divBound}
    implies that $\|d_i \|  \le 1$. Hence,
  \[
    h_{i+1} \le \frac{1}{2} h_i \|d_i\|_{\infty}^2  \le   \frac{1}{2} h_i \|d_i\|^2 \le  \frac{1}{2} h_i ( \|d_i\|+1) h_i,
  \]
    where the first inequality is \Cref{thm:newtonstep}~\ref{item:thm:newton:full} and the last
    is~\Cref{lem:divBound}. Since $\|d_i\|  \le 1$, we further conclude
    that $h_{i+1} \le h^2_i$, and that $h_{i+1} < \frac{1}{2}$. By induction,  
    $h_{i+1} \le h^2_i$ must hold for all $i$, which implies that  $h_i \le (h_0)^{2^i}$ as claimed.
  \end{proof}
\end{thm}
Using the results of this section, we can now concretely instantiate
the template log-domain interior-point method (\Cref{sec:ipm_template}).
The next section will state two algorithms.

\section{Algorithms}~\label{sec:alg}
\begin{figure}
  \centering
  \begin{minipage}[t]{.45\linewidth}~\begin{algorithm}[H]
  \SetAlgoLined\DontPrintSemicolon{}
  \SetKwFunction{algo}{shortstep}\SetKwFunction{proc}{Center}
  \SetKwProg{myalg}{Procedure}{}{}
  \myalg{\algo{$v_0,\mu_0,\mu_f$}}{
  \vspace{.1cm}
  $v \leftarrow v_0$, $\mu \leftarrow \mu_0$\\
  \While{$\mu  > \mu_f$} 
  {
    $\mu \leftarrow \frac{1}{k}\mu $ \\
    \For{$i = 1, 2, \ldots, N$}
    {
      $v \leftarrow v  + \newton(v,  \mu)$ \\
    }
 }
    \Return{$(v, \mu)$}  \\
  } 
  \end{algorithm}\end{minipage}~\begin{minipage}[t]{.65\linewidth}~\begin{algorithm}[H]
\vspace{-.35cm}
  \SetAlgoLined\DontPrintSemicolon{}
  \SetKwFunction{algo}{longstep}\SetKwFunction{proc}{Center}
  \SetKwProg{myalg}{Procedure}{}{}
  \myalg{\algo{$v_0, \mu_0, \mu_f$}}{
  \vspace{.1cm}
  $v \leftarrow v_0$, $\mu \leftarrow \mu_0$\\
    \While{$\mu  > \mu_f$ {\bf or} $\|d(v, \mu)\|_{\infty} > 1$} 
  {
    $\mu \leftarrow \min(\mu, \inf \{ \mu > 0 : \|d(v, \mu)\|_{\infty} \le 1 \})$ \\
    $\alpha \leftarrow \max(1, \frac{1}{2\beta}{\|d(v, \mu)\|^2_{\infty}})$ \\
    $v \leftarrow v + \frac{1}{\alpha} d(v, \mu)$\\
  }
    \Return{$(v, \mu)$}\\
}
\end{algorithm}
  \end{minipage}
  \caption{Algorithms for finding an approximate solution $x(v, \mu)$ to the QP~\eqref{qp:main}.
  For $(k, N)$ and $(v_0, \mu_0)$ specified by Theorem~\ref{thm:barrier}, the algorithm~{\tt shortstep}
  stays within the quadratic-convergence region of Newton's 
  method~(\Cref{thm:newtonConv}) and terminates
  in $\bigO( \sqrt m \log (\mu_0 \mu^{-1}_f) )$ Newton steps.
  For any step-size parameter $\beta \in [\frac{1}{2}, 1)$,  
 the algorithm {\tt longstep} 
  terminates given arbitrary initialization points (\Cref{thm:longstep}),
 and allows for construction of a $\mu_f$-sub-optimal feasible-point $x(v, \mu)$.
 }\label{fig:algorithms}
\end{figure}
The analysis of  Newton's method (\Cref{sec:Newton}) yields two
concrete IPMs (Figure~\ref{fig:algorithms}) for solving the QP~\eqref{qp:main}. 
The first is a \emph{short-step} algorithm:
it conservatively updates $\mu$, takes full Newton steps, and never
leaves the quadratic-convergence region of Newton's method.
The next is a \emph{long-step} algorithm: it aggressively updates $\mu$
via line-search and takes potentially damped steps.
The first algorithm, {\tt shortstep}, has an $\bigO(\sqrt m)$ iteration
bound, which is typical for interior-point methods.  The second algorithm, {\tt
longstep}, is intended for practical implementation.

\subsection{Short-step algorithm}
The algorithm {\tt shortstep} reduces the centering parameter $\mu$ by a
fixed-factor $k$ after every $N$ Newton steps. With proper selection of $k$
and $N$, it updates a given centered-point $\hat v(\mu_0)$  to $\hat v(\mu_f)$ 
using at most $C  \sqrt m \log (\mu_0 {\mu_f}^{-1})$
iterations, where $C$ is an explicit constant.  If the quadratic objective term
is zero ($W = 0$), it reduces to the short-step algorithm
of~\cite[Section 2]{permenter2020geodesic}. Its analysis is also identical once
we show that a key divergence bound still holds for non-zero $W$.

To begin, let $q(t) := 2 (\cosh(t) - 1)$. The next lemma establishes the
aforementioned bound and reduces to~\cite[Theorem 3.1]{permenter2020geodesic}
when $W = 0$.
\begin{lem}\label{lem:cpdist}
  For all $\mu_1, \mu_2 > 0$, the centered points $\hat v(\mu_1)$ and $\hat v(\mu_2)$
  satisfy
\[
  \frac{1}{m} h(\hat v(\mu_1), \hat v(\mu_2)) \le q(\frac{1}{2} \log  \frac{\mu_1}{\mu_2})
\]
\end{lem}
\begin{proof}
  Let $w_1 = e^{\hat v(\mu_1)}$ and $w_2 = e^{\hat v(\mu_2)}$
  and let $k_1 = \sqrt{\mu_1}$ and $k_2 = \sqrt{\mu_2}$.
  By primal-dual feasibility,
\[
   b = k_1 w^{-1}_{1} - Ax_1 = k_2 w^{-1}_{2} - Ax_2.
\]
Taking inner-products with $w_1$ and $w_2$ gives:
\[
  \frac{ k_1 m + w^T_1 A(x_2-x_1)}{k_2} =  w^T_1 w^{-1}_{2} , \qquad \frac{ k_2 m + w^T_2 A(x_1-x_2)}{k_1} =  w^T_2 w^{-1}_{1} 
\]
Adding  and simplifying yields:
  \begin{align*}
    h(\hat v(\mu_1), \hat v(\mu_2)) + 2m  &:= w^T_1 w^{-1}_{2}  + w^T_2 w^{-1}_{1}  \\
                                          &=  \frac{(k^2_2 + k^2_1) m + (k_1w^T_1  - k_2w^T_2 ) A(x_2-x_1)}{k_1 k_2} \\
                                          &= \frac{(k^2_2 + k^2_1)  m + (c + Wx_1   -(c+ W x_2) )^T(x_2-x_1)}{k_1 k_2} \\
                                          &= \frac{(k^2_2 + k^2_1)  m + (W(x_1-x_2) )^T(x_2-x_1)}{k_1 k_2} \\
                                          &= \frac{(k^2_2 + k^2_1)  m - (x_1-x_2)^{T}W(x_2-x_1)}{k_1 k_2} \\
                                          &\le \frac{(k^2_2 + k^2_1)  m }{k_1 k_2} 
                                          = m(\frac{k_2}{k_1} +  \frac{k_1}{k_2})
                                          = 2m \cosh(\log \frac{k_1}{k_2}) 
                                          = 2m \cosh(\log \sqrt{\frac{\mu_1}{\mu_2}}) 
  \end{align*}
Subtracting $2m$ proves the claim.
\end{proof}
\noindent We also need a lemma from~\cite{permenter2020geodesic} that relates
divergence to Euclidean distance.
\begin{lem}[Lemma 3.2 of~\cite{permenter2020geodesic}]\label{lem:divanddist}
  For all $v_1, v_2 \in \mathbb{R}^m$ it holds that $\|v_1 - v_2\|^2 \le  h(v_1, v_2) \le q(\|v_1-v_2\|)$.
\end{lem}
\noindent Using these lemmas, the $\mu$-update factor $k$ and the inner-iteration count $N$
are selected such that the divergence $h_{\mu}(v)$ remains
bounded at each iteration by a specified $\theta \in (0, \frac{1}{2}]$. 
This in turn implies that each Newton step is
quadratically convergent (\Cref{thm:newtonConv}).  
To ensure that $h_{\frac{1}{k} \mu}(v) \le \theta$  just before $\mu$ updates,
 we use the following upper-bound
\begin{align}\label{eq:tri}
  h_{\frac{1}{k} \mu}(v)   \le q( \|v-\hat v(\mu) \| +  \|  \hat v(\mu) -\hat v(\frac{1}{k}\mu) \|),
\end{align}
which follows from~\Cref{lem:divanddist} and the triangle inequality.
For a specified $\epsilon$, the parameter $N$ is chosen to ensure that
$\|v-\hat v(\mu) \| \le \epsilon$ using \Cref{thm:newtonConv} and
\Cref{lem:divanddist}.  Using Lemma~\ref{lem:cpdist} and \Cref{lem:divanddist},
the parameter $k$ is then chosen to ensure that
$\|  \hat v(\mu) -\hat v(\frac{1}{k}\mu) \| \le q^{-1}(\theta) - \epsilon$,
where $q^{-1} : \mathbb{R} \rightarrow \mathbb{R}_{+}$ denotes
the nonnegative inverse of $q$.
Together with~\eqref{eq:tri}, this implies that  $h_{\frac{1}{k} \mu}(v) \le \theta$. 

A formal statement of the $(k, N)$-selection criteria and the convergence
guarantees of {\tt shortstep} follow. We omit proof, as it is identical
to~\cite[Theorem~2.1]{permenter2020geodesic}.
\begin{thm}~\label{thm:barrier}
  Let  {\tt shortstep} (\Cref{fig:algorithms}) have parameters $(k, N)$ that satisfy, for
  some  $\frac{1}{2} \ge  \theta > 0$ and $\qup^{-1}(\theta) > \epsilon > 0$, the conditions
  \begin{align}\label{eq:mainass}
    \theta^{2^N}  \le \epsilon^2, \qquad \frac{1}{2} \log k = \qup^{-1}(\frac{1}{m} \zeta^2),
  \end{align}
  where $\zeta := \qup^{-1}(\theta) - \epsilon$.   
  The following statements hold for {\tt shortstep} given input $(\hat v(\mu_0), \mu_0, \mu_f)$: 
  \begin{enumerate}[label= (\alph*)]
      \item At most  $N \lceil c^{-1} \sqrt{m} \log \frac{\mu_0}{\mu_f} \rceil$ Newton steps execute,
        where $c := 2\qup^{-1}(\zeta^2)$.
      \item  The output $(v, \mu)$ satisfies $\|v-\hat v(\mu)\| \le \epsilon$ and $\mu \le \mu_f$.
    \end{enumerate}
\end{thm}
\noindent 
Observe that {\tt shortstep} takes only full Newton-steps  and that
its convergence guarantees assume a centered initialization point, i.e., that $v_0 = \hat v(\mu_0)$. 
The next algorithm {\tt longstep} will support arbitrary initialization through the use
of damped Newton steps.

\subsection{Long-step algorithm}~\label{sec:longstep}
The procedure {\tt longstep}~(\Cref{fig:algorithms})  supports
arbitrary initialization and performs more aggressive updates of 
the centering parameter $\mu$.
At each
iteration, it finds the smallest $\mu$ for
which the Newton direction $d(v, \mu)$ satisfies $\|d(v, \mu)\|_{\infty} \le 1$, if
such a $\mu$ exists.  It terminates once both $\|d(v, \mu)\|_{\infty} \le 1$ and $\mu \le \mu_f$.
The condition $\|d(v, \mu)\|_{\infty} \le 1$ is motivated by the following lemma,
which shows that under this condition, the approximate solution $x(v, \mu)$
associated with the Newton direction (Definition~\ref{defn:NewtonDirection})
is feasible and has  bounded sub-optimality.
\begin{lem}\label{lem:feas}
  For $\mu > 0$ and $v \in \mathbb{R}^m$, let $d = \newton(v, \mu)$ and $x = x(v, \mu)$.  
  Let  $\lambda = \sqrt\mu(e^v + e^v \circ  d)$ and $s = \sqrt\mu(e^{-v} - e^{-v} \circ  d)$.
  If $\|d\|_{\infty} \le 1$,  then $(x, s, \lambda)$ satisfies the primal-dual feasibility conditions
\begin{align*}
Ax + b = s,  \;\;A^T \lambda = Wx + c, \;\; \lambda \ge 0, \;\;s \ge 0.
\end{align*}
  Further,  $\|s \circ \lambda\|_1   = \mu (m - \|d\|^2)$.
  \begin{proof}
  The equality constraints hold by definition of $\newton(v, \mu)$.
  Nonnegativity of both $s$ and $\lambda$ hold by their definition and the
    fact that $\|d\|_{\infty} \le 1$.
  Finally, since $s, \lambda \ge 0$, we have that $\|\lambda \circ s\|_1 = \langle s, \lambda \rangle$.
    Expanding $\langle s, \lambda \rangle$ using the definition of $s$ and $\lambda$
    proves the claim.

   %

  \end{proof}
\end{lem}
\noindent It remains to show that the algorithm will actually terminate.
To prove this, we note that the algorithm,
by design, monotonically decreases $\mu$ and always 
applies an update $v \leftarrow v + \frac{1}{\alpha}d(v, \mu)$
with $\|d(v, \mu)\|_{\infty} \ge 1$. We will show 
that infinitely many iterations, and the resulting 
convergence of $\mu$, contradicts $\|d(v, \mu)\|_{\infty} \ge 1$.
Our analysis also selects the step-size parameter $\beta$ 
to ensure global convergence (Theorem~\ref{thm:globalconv}) and full Newton steps  
when $\|d(v, \mu)\|_{\infty} \le 1$.
\begin{thm}~\label{thm:longstep}
  For any step-size parameter $\beta \in [\frac{1}{2}, 1)$ 
  and input $(v_0, \mu_0, \mu_f) \in \mathbb{R}^m \times \mathbb{R} \times \mathbb{R}$ with $\mu_0 > \mu_f > 0$,
  the algorithm {\tt longstep} terminates and returns $(v, \mu)$ with $\mu \le \mu_f$.
  Further, letting $x$ denote $x(v, \mu)$, we have that
  \[
    Ax+ b \ge 0,    \qquad \frac{1}{2} x^T  W x + c^{T}x \le V^* +  \mu  m,
  \]
  where $V^*$ denotes the optimal value of QP~\eqref{qp:main}.
  \begin{proof}

    Let $v_i$, $\mu_i$ and $\alpha_i$ denote the sequences generated
    by {\tt longstep} indexed by $i$ such that
    $v_{i+1} = v_i + \frac{1}{\alpha_i}d(v_i, \mu_{i})$.
    Now suppose the algorithm does not terminate.
    We first consider the case where $\mu$ updates only finitely many times.
    In this case, there exists an $M$ such that 
      $\mu_i = \mu_{M}$ for all iterations $i > M$,
      which implies that $v_i$ converges to $\hat v(\mu_M)$ by Theorem~\ref{thm:globalconv}.
      But by selection of $\mu_i$, we also have  that $\|d(v_i, \mu_i)\|_{\infty} \ge 1$ for all $i$,
    a contradiction. Hence, in this case, the algorithm must terminate.

    Now suppose that $\mu$ updates infinitely many times
    and  let $\sigma_k$ denote the subsequence of iterations $i$
    where $\mu$ changes, i.e., $\mu_{\sigma_{k}} <  \mu_{\sigma_{k-1}}$
    and $\mu_i = \mu_{\sigma_{k-1}}$ for $  \sigma_{k} > i \ge  \sigma_{k-1}$.
    For brevity, let $h_{k} : \mathbb{R}^m \rightarrow \mathbb{R}$ denote the 
    divergence $h(v,  \hat v(\mu_{\sigma_k}))$ as a function of $v$.
    We note that
    \begin{align}~\label{eq:convProof}
      \frac{1}{2} h_{k}(v_{\sigma_k}) \ge h_{k}(v_{1 + \sigma_k})   \ge        h_{k}(v_{\sigma_{k+1}}),
    \end{align}
    where the first inequality holds by Theorem~\ref{thm:newtonstep} since
    $\|d(v_i, \mu_i)\|_{\infty} = 1$ for $i = \sigma_k$ and the second because
    Newton iterations decrease $h_{\mu}(v_i)$ for fixed $\mu$.

    Since $\mu_{\sigma_k}$ is bounded below and monotonically decreasing, it converges. 
    Since the map $\mu \mapsto \hat v(\mu)$ is continuous,
    the sequence $\hat v(\mu_{\sigma_{k+1}})$ also converges.
    Hence, for any $\epsilon > 0$, there exists an $N$ such that 
    for all $k > N$,
    \[
      \|\hat v(\mu_{\sigma_k}) - \hat v(\mu_{\sigma_{k+1}})\| < \epsilon.
    \]
   This shows that for $k > N$,
    \begin{align*}
      h_{k+1}(v_{\sigma_{k+1}}) + 2m 
                                            &=  2 \sum^m_{j=1} \cosh( [v_{\sigma_{k+1}} - \hat v(\mu_{\sigma_{k+1}})]_j  ) \\
                                            &\le  2 \sum^m_{j=1} \cosh( |[v_{\sigma_{k+1}} - \hat v(\mu_{\sigma_{k}})]_j|+ \epsilon ) \\
                                            &\le 2 \sum^m_{j=1} \cosh( [v_{\sigma_{k+1}} - \hat v(\mu_{\sigma_{k}})]_j ) (\cosh(\epsilon) + \sinh(\epsilon)) \\
                                            & =   (h_{k}(v_{\sigma_{k+1}}) + 2m)  (\cosh(\epsilon) + \sinh(\epsilon)) \\
                                            & \le   (\frac{1}{2} h_{k}(v_{\sigma_{k}}) + 2m)  (\cosh(\epsilon) + \sinh(\epsilon)),
    \end{align*}
    where the third line uses the inequality $\cosh(x+y) \le \cosh(x)( \cosh(y) + \sinh(|y|))$
    and the last uses~\eqref{eq:convProof}. 
    Noting that $e^\epsilon = \cosh(\epsilon) + \sinh(\epsilon)$, we rearrange
    the last line to conclude that 
    \[
      h_{k+1}(v_{\sigma_{k+1}}) \le  \frac{e^\epsilon}{2} h_{k}(v_{\sigma_{k}}) + 2m(e^\epsilon-1).
    \]
    This shows that $h_{k}(v_{\sigma_{k}})$ is upper-bounded by 
    a sequence of the form $a_{k+1} = c_1(\epsilon) a_{k} + c_2(\epsilon)$,
    which, if $|c_1| < 1$,  converges to $L = \frac{c_2}{1-c_1}$.
    For any $\delta > 0$, we can pick $\epsilon$ small enough to show that $L < \delta$.
    This shows that $h_{k}(v_{\sigma_{k}})$ converges to zero, contradicting $\|d(v_{\sigma_{k}}, \mu_{\sigma_{k}})\| \ge 1$
    by \Cref{lem:divBound}.  Hence, the algorithm must terminate.
    Finally, the feasibility and suboptimality guarantees for $x(v, \mu)$
    follow from \Cref{lem:feas} and weak duality.
  \end{proof}
\end{thm}
\noindent We conclude with three topics related to selection of $\mu$.
\subsubsection{Computing the infimum}\label{sec:infComp}
Fix $v \in \mathbb{R}^m$ and let $d(\mu)$ denote $d(v, \mu)$.  
In this notation,
each iteration of {\tt longstep} requires computation of
$\inf \{ \mu > 0 : \|d(\mu)\|_{\infty} \le 1 \}$.
This is straight-forward upon recognition that $d(\mu)$ is an affine function
of $(\sqrt\mu)^{-1}$, i.e., it decomposes as $d(\mu) = d_0 + (\sqrt\mu)^{-1} d_1$
for some fixed $d_0\in \mathbb{R}^m$ and $d_1\in \mathbb{R}^m$.
We give a constructive proof of this fact that demonstrates how to
build this decomposition.
\begin{prop}\label{prop:ddecomp}
  There exists $d_0, d_1 \in \mathbb{R}^m$ satisfying
    $\newton(\mu) =  d_0  + \frac{1}{\sqrt \mu} d_1$ for all $\mu > 0$.
 \begin{proof}
 Fix $\mu_1 > 0$ and $\mu_2 > 0$ and $v \in \mathbb{R}^m$.
 Let $\hat d_i = \newton(v, \mu_i)$ and $k_i = \frac{1}{\sqrt\mu_i}$ for $i \in {1, 2}$.
   By Definition~\ref{defn:NewtonDirection}, there exists $\hat x_1$ and $\hat x_2$ satisfying
   \[
     A^T (e^v+ e^v \circ \hat d_1) = W \hat x_1 + k_1 c, \qquad A^T (e^v+ e^v \circ \hat d_2) = W \hat x_2 + k_2 c.
   \]
   Multiplying these equations by $t$ and $(1-t)$, respectively,
   and adding yields
   \[
     A^T (e^v+ e^v \circ (t \hat d_1 + (1-t) \hat d_2 )) = W (t \hat x_1  + (1-t)\hat x_2 ) + (t k_1 + (1-t) k_2) c.
   \]
   By similar argument,
   \[
      e^{-v} -  e^{-v} \circ  (t \hat d_1 + (1-t) \hat d_2 )   = A(t \hat x_1  + (1-t)\hat x_2 ) +  (t k_1 + (1-t) k_2)  b.
  \]
    By Definition~\ref{defn:NewtonDirection}, we conclude that $t \hat d_1 + (1-t) \hat d_2 = d(v, \mu)$ for
    $\frac{1}{\sqrt\mu} = t k_1 + (1-t) k_2$. Solving for $t$ shows that
     $t  = c_1 \frac{1}{\sqrt\mu} + c_0$ for $c_1 = (k_1-k_2)^{-1}$ and $c_0 = -k_2(k_1-k_2)^{-1}$.
   Substituting into $t \hat d_1 + (1-t) \hat d_2$, we deduce that
   \[
     d(v, \mu) = (c_1 \frac{1}{\sqrt\mu} + c_0) \hat d_1 + (1-c_1 \frac{1}{\sqrt\mu} - c_0) \hat d_2.
   \]
   Hence, the claim follows for $d_0 = c_0 \hat d_1 + (1-c_0) \hat d_2$
   and $d_1 = c_1 (\hat d_1 - \hat d_2)$.
  \end{proof}
\end{prop}
This decomposition in turn allows us to  characterize the condition
$\|d(\mu)\|_{\infty} \le 1$ using a system of linear inequalities
immediate from the definition of $\|\cdot\|_{\infty}$.
  \begin{prop}~\label{prop:mustar}
    For $d_0, d_1 \in \mathbb{R}^m$, 
    we have that $\|d_0 + \frac{1}{\sqrt\mu} d_1\|_{\infty} \le 1$ if and only if
    \begin{align}\label{eq:mustar}
      -\ones \le d_0 +  \frac{1}{\sqrt\mu} d_1 \le \ones ,
  \end{align}
    where $\ones \in \mathbb{R}^m$ denotes the vector of all ones.
  \end{prop}
\noindent Note that minimizing $\mu$ subject to these inequalities can be done
in $\bigO(m)$ time simply by iterating over the components of $d_0, d_1 \in \mathbb{R}^m$.
\subsubsection{Reuse of factorizations}
The constructive proof of \Cref{prop:ddecomp}  
builds the decomposition $d_0 + (\sqrt\mu)^{-1} d_1$
from two Newton directions $d(v, \mu_1)$ and $d(v, \mu_2)$.  Since $v$ is \emph{fixed}, 
these directions are found by solving 
two Newton systems (\Cref{defn:NewtonDirection}) with the \emph{same} 
positive definite coefficient matrix $W + A^T Q(v) A$.  
Hence, one can find both directions using the same Cholesky factorization
of $W + A^T Q(v) A$.
\subsubsection{Least-squares $\mu$}\label{sec:least_squares}
The decomposition $d(\mu) = d_0 + (\sqrt\mu)^{-1} d_1$
from \Cref{prop:ddecomp}
also enables easy computation of the least-squares $\mu$, i.e., the $\mu$ 
that minimizes $\|d(\mu)\|^2$. 
Assuming $d^T_0 d_1 < 0$, this $\mu$ is the 
solution of  $-d^T_0 d_1  \sqrt\mu =\|d_1\|^2$
   and provides a natural heuristic choice for the initial $\mu_0$ passed to {\tt longstep}.

\section{Comparison with barrier methods}\label{sec:barrier}
The log-domain update $v \leftarrow v + d(v, \mu)$ 
induces  multiplicative updates 
$s \leftarrow s \circ e^{-d}$ and $\lambda \leftarrow  \lambda \circ e^{d}$  of the slack
variable $s:= \sqrt\mu e^{-v}$ and the Lagrange multiplier $\lambda := \sqrt\mu e^{v}$.
Taylor expanding $e^{-d}$ and $e^{d}$ yields approximations
of these updates:
\begin{align}~\label{eq:FirstOrderApprox}
 s \circ e^{-d}\approx s \circ (\ones-d), \quad \lambda \circ e^{d}  \approx \lambda \circ (\ones + d).
\end{align}
In this section, we interpret these approximations in the context of \emph{barrier methods}.
Specifically, we show that $s \leftarrow s \circ (\ones-d)$ is 
equivalent to an iteration of the \emph{primal barrier method}
applied to the QP~\eqref{qp:main}. Similarly, we show that
$\lambda  \leftarrow \lambda \circ (\ones + d)$ is 
equivalent to an iteration of the \emph{dual barrier method} applied
to the dual QP~\cite{dorn1960duality}, which takes the form
\begin{align*}
\maximize_{u, \lambda} & -(\frac{1}{2} u^T W u  +b^T \lambda) \;\; \mbox{subject to }  A^T\lambda = Wu + c, \lambda \ge 0.
\end{align*}
Computational experiments illustrate that replacing the log-domain
update with either of these approximations increases
the number of {\tt longstep} iterations needed to solve random QPs
to  fixed accuracy,
illustrating in effect that barrier methods are less efficient than
log-domain IPMs.

\subsection{Primal barrier method}
For $\mu > 0$, the primal barrier method applies Newton's method to the
optimality conditions of
\begin{align*}
\minimize_{x,s} \;\; & \frac{1}{2}x^T W x + c^{T} x - \mu \sum^m_{i=1} \log s_i \;\;  \mbox{subject to }  s = Ax + b.
\end{align*}
These conditions read $A^T \lambda = Wx + c$, $s=Ax + b$ and $\mu s^{-1} = \lambda$,
where $\lambda$ is a Lagrange multiplier for the equality constraints. 
Letting $z:=(x, s, \lambda)$, Newton iterations take the form $z \leftarrow z + \Delta z$,
where $\Delta z := (\Delta x, \Delta s, \Delta \lambda)$
solves
\begin{align}\label{eq:PrimalBarrierIter}
\begin{aligned}
A^T(\lambda + \Delta \lambda) = W (x + \Delta x) + c, \;\; s + \Delta s = A(x+\Delta x) + b,\\ 
\mu (s^{-1} - s^{-2} \circ \Delta s) = \lambda + \Delta \lambda.
\end{aligned}
\end{align}
The following shows that the Newton update of $s$ is precisely equivalent
to the primal linearization~\eqref{eq:FirstOrderApprox} of the log-domain Newton step,
i.e., it is equivalent to replacing $s \circ e^{-d}$ with its first-order approximation $s \circ e^{-d} \approx s \circ (\ones - d)$.
\begin{prop}\label{prop:PrimalBarrier}
Consider $(x, s, \lambda) \in \mathbb{R}^n \times \mathbb{R}^m \times \mathbb{R}^m$ with $s > 0$.
Let $(\Delta x, \Delta s, \Delta \lambda)$ solve~\eqref{eq:PrimalBarrierIter}.
Finally, let $v \in \mathbb{R}^m$ and $\mu > 0$ satisfy $s = \sqrt\mu e^{-v}$.
Then $s + \Delta s = s \circ(\ones - d(v, \mu))$, where $d(v, \mu)$ is the log-domain
Newton direction (\Cref{defn:NewtonDirection}).
\begin{proof}
Letting $\hat x = x + \Delta x$, the conditions~\eqref{eq:PrimalBarrierIter} simplify to
\[
\mu A^T ( s^{-1} - s^{-2} \circ \Delta s)= W\hat x + c, \;\; s + \Delta s =A\hat x + b.
\]
Letting $\hat d := -s^{-1} \circ \Delta s$, we conclude that
\[
 \mu A^T  [s^{-1}\circ (\ones  + \hat d )]= W\hat x + c,  \;\; s \circ( \ones  - \hat d)  =Ax + b.
\]
Substituting  $s^{-1} = \sqrt\mu^{-1} e^{v}$ and $s= \sqrt\mu e^{-v}$ yields
\[
\sqrt\mu A^T[e^v \circ (\ones  + \hat d )] = W \hat x + c, \;\;  
\sqrt\mu e^{-v} \circ( \ones  - \hat d)  =A \hat x + b,
\]
which is precisely the definition of $d(v, \mu)$.
Hence, $d(v, \mu) = \hat d$ since $d(v, \mu)$ is unique by~\Cref{thm:NewtonChar}.
By definition of $\hat d$, we also have that $\Delta s = -s \circ \hat d$, proving
the claim.
\end{proof}
\end{prop}
\noindent We next show an analogous interpretation holds for the
dual linearization $\lambda \circ e^d \approx \lambda \circ (\ones + d)$.
\subsection{Dual barrier method}
Given $\mu > 0$, the dual barrier method applies Newton's method to the
optimality conditions of
\begin{align*}
\minimize_{u, \lambda} \;\;& \frac{1}{2} u^T W u + b^T \lambda   - \mu \log \lambda \;\;  \mbox{  subject to }  A^T \lambda = Wu + c.
\end{align*}
These optimality conditions read
 $Wu =  W \gamma$, $A^T \lambda = Wu + c$ and $A \gamma + b = \mu \lambda^{-1}$,
 where $\gamma$ is a Lagrange multiplier for the equality constraints.
Letting $z :=(\gamma, u, \lambda)$, 
Newton iterations take the form $z \leftarrow z + \Delta z$,
where $\Delta z := (\Delta \gamma, \Delta u, \Delta \lambda)$
solves 
\begin{align}\label{eq:dualBarrierDef}
\begin{aligned}
W(u+\Delta u) = W (\gamma + \Delta \gamma), \;\;
A^T (\lambda + \Delta \lambda) = W(u+\Delta u) + c  \\
A (\gamma + \Delta \gamma) + b = \mu( \lambda^{-1} - \lambda^{-2} \circ \Delta \lambda).
\end{aligned}
\end{align}
The following shows that the Newton update of $\lambda$ is precisely equivalent
to the dual linearization~\eqref{eq:FirstOrderApprox} of the log-domain Newton step, i.e.,
it is equivalent to replacing $\lambda \circ e^d$ with its first-order approximation $\lambda \circ e^d \approx \lambda \circ (\ones + d)$.
\begin{prop}\label{prop:DualBarrier}
Consider $(\gamma, u, \lambda) \in \mathbb{R}^n \times \mathbb{R}^n \times \mathbb{R}^m$ with $\lambda > 0$.
Let $\Delta z := (\Delta \gamma, \Delta u, \Delta \lambda)$ solve~\eqref{eq:dualBarrierDef}.
Finally, let $v \in \mathbb{R}^m$ and $\mu > 0$ satisfy $\lambda = \sqrt\mu e^{v}$.
Then $\lambda + \Delta \lambda = \lambda \circ (\ones +  d(v, \mu))$, where $d(v, \mu)$ is the 
log-domain Newton direction (\Cref{defn:NewtonDirection}).
\begin{proof}
Letting $\hat \gamma = \gamma + \Delta \gamma$, the conditions~\eqref{eq:dualBarrierDef}
simplify to
\begin{align*}
A \hat \gamma + b = \mu (\lambda^{-1} - \lambda^{-2} \circ \Delta \lambda), \;\;  A^T (\lambda + \Delta \lambda) = W\hat \gamma + c.
\end{align*}
Substituting $\lambda = \sqrt\mu e^v$
and letting $\hat d:= \lambda^{-1} \circ \Delta \lambda$,
we obtain
\[
A \hat \gamma + b = \sqrt\mu e^{-v}\circ(\ones - \hat d), \qquad   \sqrt\mu A^T [e^v\circ(\ones + \hat d)] = W\hat \gamma + c,
\]
which is precisely the definition of $d(v, \mu)$.
Hence, $d(v, \mu) = \hat d$ since $d(v, \mu)$ is unique by~\Cref{thm:NewtonChar}.
By definition of $\hat d$, we have that $ \Delta \lambda = \lambda  \circ \hat d$, proving
the claim.
\end{proof}
\end{prop}

\subsection{Computational comparison}
We next give simple computational experiments that compare
barrier methods with {\tt longstep}. These experiments
show that barrier methods require more iterations to reach a target
duality gap when initialized at identical starting points. Code for reproducing
these experiments is located at 
\[
{\tt https://github.com/frankpermenter/LDIPMComputationalResults}
\]
\paragraph{Barrier method implementations}

We invoke \Cref{prop:PrimalBarrier}
and implement the primal barrier method by modifying a single
line of {\tt longstep}. Precisely, we replace 
$v \leftarrow v + \alpha^{-1}d$ 
with the approximation 
$v \leftarrow -\log [e^{-v}\circ (\ones-\alpha^{-1} d)]$,
which is equivalent to taking $s \leftarrow s + \alpha^{-1}\Delta s$,
with $(s, \Delta s)$ as defined in \Cref{prop:PrimalBarrier}.
We similarly implement the dual barrier method
using \Cref{prop:DualBarrier}.  That is, we replace 
$v \leftarrow v + \alpha^{-1}d$ 
with $v \leftarrow\log [e^{v}\circ (\ones+\alpha^{-1} d)]$,
which is equivalent to taking $\lambda \leftarrow \lambda +
\alpha^{-1}\Delta \lambda$, with $(\lambda, \Delta \lambda)$ as defined in \Cref{prop:DualBarrier}.
We also slightly modify selection of $\mu$, replacing the $\|d(v, \mu)\|_{\infty} \le 1$
bound with $\|d(v, \mu)\|_{\infty} \le 1- \epsilon$:
since $\alpha \ge 1$, this ensures that the argument to the $\log$ function is always positive.

\paragraph{Instances}
Our comparison uses randomly generated QPs that satisfy the regularity
conditions of Assumption~\ref{ass:main}.  The inequality matrix $A$ 
has entries drawn from a normal distribution with zero mean
and unit variance. Each row is then rescaled to have unit norm.
The cost matrix $W$ is constructed as $W = R^T R$ with $R \in \mathbb{R}^{r \times n}$
sampled and normalized the same way as $A$. Finally $c$ and $b$ are chosen
by sampling $x$, $\lambda > 0$ and $s > 0$  and setting $b = s - Ax$,  $c = A^{T}\lambda  - Wx$.
The vector $x$ is drawn from a normal distribution.  The vector $s$  is chosen
as $\ones + \frac{1}{10} |w|$, where $w$ is 
also sampled from a normal distribution. Finally, $\lambda$ is chosen the same way as $s$
using a different random $w$.

\paragraph{Results}
\Cref{tab:results} shows superior performance of {\tt longstep} on
a set of instances. In this set, we vary either the rank of $W \in \mathbb{R}^{n \times n}$ or the number of inequalities $m$, i.e.,
the number of rows of $A \in \mathbb{R}^{m \times n}$. The rank of $W$ is controlled by the
number of columns $r$ of $R \in \mathbb{R}^{r \times n}$, recalling that  $W = R^T R$.
 The algorithms are initialized with $v_0 = 0$,   $\mu_f = 10^{-3}$,
 and $\mu_0 = \mu_{ls}$, where $\mu_{ls}$ denotes the least-squares $\mu$ (\Cref{sec:least_squares}).

\begin{table}
  \centering
  \small{}
  \begin{tabular}{rrrrr}
  && \multicolumn{3}{c}{Iterations}\\
    $m$   &  $\rank W$ & LS & DB & PB  \\
    \toprule
 200 &   0 &   8.9 & 9.5 & 9.5 \\
 200 &  50 &   7.1 & 8.3 & 7.6 \\
 200 & 100 &   6.5 & 7.9 & 7.1 \\
\midrule
 100 &  50 &   6.3 & 7.1 & 7.3 \\
 150 &  50 &   6.8 & 7.7 & 7.6 \\
 200 &  50 &   7.1 & 8.3 & 7.6 \\

  \end{tabular}
  \qquad
  \begin{tabular}{rrrrr}
  && \multicolumn{3}{c}{Iterations}\\
    $m$   &  $\rank W$ & LS & DB & PB  \\
    \toprule
 2000 &   0 &   10.8 & 11.6 & 11.7 \\
 2000 & 500 &   8.1 & 9.8 & 8.8 \\
 2000 & 1000 &   7.5 & 9.0 & 8.0 \\
\midrule
 1000 & 500 &   7.3 & 8.0 & 8.7 \\
 1500 & 500 &   7.9 & 9.0 & 8.8 \\
 2000 & 500 &   8.1 & 9.8 & 8.8 \\
  \end{tabular}
  \caption{Average iterations on 30 random QPs with $m$ inequalities in
  $n=100$ variables (left) and $n=1000$ variables (right). Iterations needed to
  reach a target duality-gap are shown for {\tt longstep} (LS), the
  dual barrier method (DB), and the primal barrier method (PB).
   }\label{tab:results}
\end{table}


{\small
\bibliographystyle{abbrvnat}
\bibliography{bib}
}
\end{document}